\documentclass[a4paper, 12pt]{article}
\usepackage{mathrsfs}
\usepackage{amsmath}
\usepackage{amsthm}
\usepackage{array}
\usepackage{cases}
\usepackage{indentfirst}

\usepackage[T1]{fontenc}
\usepackage[varg]{txfonts}
\usepackage{graphicx, epsfig, subfigure}
\usepackage{enumerate} 
\usepackage[square, numbers, sort&compress]{natbib} 
\ifx\pdfoutput\undefined
 \usepackage[dvipdfm,%
  pdfstartview=FitH, 
  bookmarks=true,%
  bookmarksnumbered=true, 
  bookmarksopen=true, 
  plainpages=false,%
  pdfpagelabels,%
  colorlinks=true, 
  linkcolor=blue, 
  citecolor=blue,%
  urlcolor=black,
  pdfborder=001]{hyperref}
  \AtBeginDvi{}  
\else
 \usepackage[pdftex,%
  pdfstartview=FitH, 
  bookmarks=true,%
  bookmarksnumbered=true, 
  bookmarksopen=true, 
  plainpages=false,%
  pdfpagelabels,%
  colorlinks=true, 
  linkcolor=blue, 
  citecolor=blue,%
  urlcolor=black,
  pdfborder=001]{hyperref}
\fi

\makeatletter
\def\th@plain{\upshape
\itshape 
} \makeatother

\makeatletter
\renewenvironment{proof}[1][\proofname]{\par
  \pushQED{\qed}%
  \normalfont \topsep6\p@\@plus6\p@\relax
  \trivlist
  \item[\hskip\labelsep
        \bfseries
    #1\@addpunct{.}]\ignorespaces
}{%
  \popQED\endtrivlist\@endpefalse
} \makeatother

\newtheorem{thm}{Theorem}[section]
\newtheorem{cor}[thm]{Corollary}
\newtheorem{lem}[thm]{Lemma}
\newtheorem{conj}[thm]{Conjecture}
\newtheorem{prop}[thm]{Proposition}

\numberwithin{equation}{section}
\def\pn{\par\noindent}
\def\qed{\hfill \rule{2.5mm}{2.5mm}}

\newcommand{\ie}{i.e.,\ }

\newcommand{\set}[1]{\left\{#1\right\}}

\newcommand{\D}{\displaystyle}
\newcommand{\DF}[2]{\frac{\D#1}{\D#2}}
\newcommand{\lf}{\left\lfloor}
\newcommand{\rf}{\right\rfloor}

\newcommand{\lk}{\left(}
\newcommand{\rk}{\right)}
\newcommand{\M}{\lf\DF{M+2}{4}\rf}
\newcommand{\F}{\cup_{x\in N(u)}\Phi(x)}
\newcommand{\E}{\cup_{e\ni u}I_\Phi(e)}

\begin{document}

\title{(2,1)-Total labeling of planar graphs with large maximum
degree \thanks{This research is partially supported by
IIFSDU(2009hw001), NNSF(61070230, 11026184, 10901097) and
RFDP(200804220001, 20100131120017) and SRF for ROCS.}}
\author{Yong Yu \thanks{Corresponding author.\quad yuyong6834@yahoo.com.cn}, Xin Zhang, Guanghui Wang, Jinbo Li \\[.5em]
{\small School of Mathematics, Shandong University, Jinan 250100, P. R. China}\\
}
\date{}

\maketitle

\begin{abstract}\baselineskip 0.60cm
The ($d$,1)-total labelling of graphs was introduced by Havet and
Yu. In this paper, we prove that, for planar graph $G$ with maximum
degree $\Delta\geq12$ and $d=2$, the (2,1)-total labelling number
$\lambda_2^T(G)$ is at most $\Delta+2$.
\\[.5em]
\end{abstract}

\baselineskip 0.60cm

\section{Introduction}
In this paper, all graphs considered are finite, simple and
undirected. We use $V(G)$, $E(G)$, $\delta(G)$ and $\Delta (G)$ (or
simply $V$, $E$, $\delta$ and $\Delta$) to denote the vertex set,
the edge set, the minimum degree and the maximum degree of a graph
$G$, respectively. Let $G$ be a plane graph. We always denote the
face set of $G$ by $F(G)$. The degree of a face $f$, denoted by
$d(f)$, is the number of edges incident with it, where cut edge is
counted twice. A $k$-, $k^+$- and $k^-$-vertex (or face) in graph
$G$ is a vertex (or face) of degree $k$, at least $k$ and at most
$k$, respectively. Furthermore, if a vertex $v$ is adjacent to a
$k$-vertex $u$, we say that $u$ is a $k$-neighbor of $v$. For $f\in
F(G)$, we call $f$ a $[d(v_1),d(v_2),\cdots,d(v_k)]$-face if
$v_1,v_2,\cdots,v_k$ are the boundary vertices of $f$ in clockwise
order. A 3-face is also usually called a triangle face. Readers are
referred to \cite{Bondy} for other undefined terms and notations.

The ($d$,1)-total labelling of graphs was introduced by Havet and Yu
\cite{Havet}. A \emph{$k$-($d$,1)-total labelling} of a graph $G$ is
a function $c$ from $V(G)\cup E(G)$ to the color set
$\{0,1,\cdots,k\}$ such that $c(u)\neq c(v)$ if $uv\in E(G)$,
$c(e)\neq c(e')$ if $e$ and $e'$ are two adjacent edges, and
$|c(u)-c(e)|\geq d$ if vertex $u$ is incident to the edge $e$. The
minimum $k$ such that $G$ has a $k$-($d$,1)-total labelling is
called the \emph{($d$,1)-total labelling number} and denoted by
$\lambda_d^T(G)$. Readers are referred to \cite{BMR,Chen,Lih,MR,WC}
for further research. When $d=1$, the (1,1)-total labelling is the
well-known total coloring of graphs. Havet and Yu gave a conjecture
similar to Total Coloring Conjecture, which is called ($d$,1)-Total
Labelling Conjecture.

\begin{conj}[\cite{Havet}]\label{conj:DTLC}
Let $G$ be a graph. Then $\lambda_d^T(G)\leq\min\{\Delta+2d-1,\
2\Delta+d-1\}$.
\end{conj}

\pn When $d=2$, ($d$,1)-Total Labelling Conjecture can be rewritten
as follows.

\vspace{0.3cm}\noindent{\bf Conjecture \ref{conj:DTLC}$^{\prime}$.}
Let $G$ be a graph. Then $\lambda_2^T(G)\leq\Delta+3$.

\vspace{0.3cm} In \cite{Chen}, Chen and Wang studied the (2,1)-total
labelling number of outerplanar graphs. In \cite{BMR}, Bazzaro,
Montassier and Raspaud proved a theorem for planar graph with large
girth and high maximum degree:

\begin{thm}[\cite{BMR}]\label{Thm:BMR}
Let $G$ be a planar graph with maximum degree $\Delta$ and girth $g$. Then $\lambda_d^T(G)\leq\Delta+2d-2$ with $d\geq2$ in the following cases:\\
(1)\ $\Delta\geq2d+1$ and $g\geq11$; (2)\ $\Delta\geq2d+2$ and
$g\geq6$; (3)\ $\Delta\geq2d+3$ and $g\geq5$; (4)\ $\Delta\geq8d+2$.
\end{thm}

\pn Additionally, the following risky conjecture was also proposed
in \cite{BMR}.

\begin{conj}[\cite{BMR}]\label{conj:BMR}
For any planar triangle-free graph $G$ with $\Delta\geq3$,
$\lambda_d^T(G)\leq\Delta+d$.
\end{conj}

Let $\chi$ and $\chi'$ denote the chromatic number and the edge
chromatic number, respectively. The following results was first
mentioned in \cite{Havet}.

\begin{prop}[\cite{Havet}]\label{prop:Havet}
Let $G$ be a graph with degree $\Delta$. Then
\begin{enumerate}
  \vspace{-2mm}\item[(1)] $\lambda_d^T(G)\leq\chi+\chi'+d-2$;
  \vspace{-2mm}\item[(2)] $\lambda_d^T(G)\geq\Delta+d-1$;
  \vspace{-2mm}\item[(3)] $\lambda_d^T(G)\geq\Delta+d$ if $d\geq\Delta$ or $G$ is $\Delta$-regular.
\end{enumerate}
\end{prop}

By (1) of Proposition \ref{prop:Havet}, for planar graph with large
maximum degree, since $\chi\leq4$ and $\chi'=\Delta$ \cite{SZH},\,
($d$,1)-Total Labelling Conjecture is meaningful only for $d$ with
$\Delta+2d-1\leq\Delta+d+2$, \ie $1\leq d\leq3$. That is why we only
consider (2,1)-total labellings for planar graph in this paper.

Our main result, shown as in Theorem \ref{Thm:2TL}, is an
improvement of Theorem \ref{Thm:BMR} when $d=2$. On the other hand,
it is also can be seen as a support for Conjecture \ref{conj:BMR}
and ($d$,1)-Total Labelling Conjecture when $d=2$. Furthermore, the
upper bound $\Delta+2$ is best possible because planar graph with
arbitrary maximum degree and $\lambda_2^T(G)=\Delta+2$ was given in
\cite{BMR}.

\begin{thm}\label{Thm:2TL}
Let $G$ be a planar graph with maximum degree $\Delta\geq12$. Then
$\Delta+1\leq\lambda_2^T(G)\leq\Delta+2$.
\end{thm}

The lower bound of our result is trivial by (2) of Proposition
\ref{prop:Havet}. For the upper bound, we prove a conclusion which
is slightly stronger as follows.

\begin{thm}\label{Thm:new2TL}
Let $M\geq12$ be an integer and let $G$ be a planar graph with
maximum degree $\Delta\leq M$. Then $\lambda_2^T(G)\leq M+2$. In
particular, $\lambda_2^T(G)\leq\Delta+2$ if $M=\Delta$.
\end{thm}

\pn The interesting case of Theorem \ref{Thm:new2TL} is when $M =
\Delta(G)$. Indeed, Theorem \ref{Thm:new2TL} is only a technical
strengthening of Theorem \ref{Thm:2TL}. But without it we would get
complications when considering a subgraph $H\subset G$ such that
$\Delta(H)<\Delta(G)$.

Let $G$ be a minimal counterexample in terms of $|V|+|E|$ to Theorem
\ref{Thm:new2TL}. By the minimality of $G$, any proper subgraph of
$G$ is (2,1)-total labelable. It is not difficult to see that $G$ is
connected. In Section 2, we obtain some structural properties of our
minimal counterexample $G$. In Section 3, we complete the proof with
discharging method.

\section{Structural properties}
From now on, we will use without distinction the terms \emph{color}
and \emph{label}. Let $X$ be a set, we usually denote the
cardinality of $X$ by $|X|$. A partial (2,1)-total labelling of $G$
is a function $\Phi$ from $X\subseteq V(G)\cup E(G)$ to the color
interval $C=\{0,1,\cdots,k\}$ with $|C|=k+1=M+3$ such that the color
of element $x\in X$, denoted by $\Phi(x)$, satisfies all the
conditions in the definition of (2,1)-total labelling of graphs.
Next, we need some notations to make our description concise.

$E_\Phi(v)=\{\Phi(e)\ |\ e\in E\ \mbox{is incident with vertex
$v$}\}$ for $v\in V$;

$I_\Phi(x)=\{\Phi(x)-1,\Phi(x),\Phi(x)+1\}\cap C$ for $x\in V\cup
E$;

$F_\Phi(v)=E_\Phi(v)\cup I_\Phi(v)$ for $v\in V$;

$A_\Phi(uv)=C\ \backslash\  \lk F_\Phi(u)\cup F_\Phi(v)\rk$ for
$uv\in E$;

$A_\Phi(u)=C\ \backslash\ \left[\lk\F\rk\cup\lk\E\rk\right]$ for
$u\in V$.

In all the notations above, only elements got colors under the
partial (2,1)-total labelling $\Phi$ are counted in our notations.
For example, if $v$ is not colored under $\Phi$, then
$F_\Phi(v)=E_\Phi(v)$ by our definition. It is not difficult to see
that $A_\Phi(uv)$ (resp. $A_\Phi(u)$) is just the set of colors
which are still available for labelling $uv$ (resp. $u$) under the
partial (2,1)-total labelling $\Phi$. Thus, if $|A_\Phi(uv)|\geq1$
(res. $|A_\Phi(u)|\geq1$), then we can (2,1)-total labelling edge
$uv$ (res. vertex $u$) properly under $\Phi$.

To prove the main result, we give the following lemmas.

\begin{lem}\label{lem:1structure}
For each $uv\in E$, we have $d(u)+d(v)\geq M-1$.
\end{lem}

\begin{proof}
Assume that there is an edge $uv\in E$ such that $d(u)+d(v)\leq
M-2$. By the minimality of $G$, $G-e$ has a (2,1)-total labelling
$\Phi$ with color interval $C$. Since
$|A_\Phi(uv)|=|C|-|F_\Phi(u)\cup
F_\Phi(v)|\geq|C|-(d(u)+d(v)-2+3\times2)\geq |C|-(M+2)\geq1$, we can
extend $\Phi$ from subgraph $G-e$ to $G$, a contradiction.
\end{proof}

\begin{lem}\label{lem:2structure}
For any edge $e=uv\in E$ with
$\min\{d(u),d(v)\}\leq\lf\DF{M+2}{4}\rf$, we have $d(u)+d(v)\geq
M+2$.
\end{lem}

\begin{proof}
Suppose there is an edge $uv\in E$ such that
$d(u)\leq\lf\DF{M+2}{4}\rf$ and $d(u)+d(v)\leq M+1$. By the
minimality of $G$, $G-e$ is (2,1)-total labelable with color
interval $C$. Erase the color of vertex $u$, and denote this partial
(2,1)-total labelling by $\Phi$. Then
$|A_\Phi(uv)|\geq|C|-|F_\Phi(u)|-|F_\Phi(v)|=|C|-|E_\Phi(u)|-|F_\Phi(v)|\geq|C|-(d(u)+d(v)-2+3)\geq
|C|-(M+2)\geq1$ which implies that $uv$ can be properly colored. We
still denote the labelling by $\Phi$ after $uv$ is colored. Next,
for vertex $u$, $|A_\Phi(u)|\geq |C|-|\F|-|\E|\geq M+3-4d(u)\geq1$.
Thus, we can extend the partial (2,1)-total labelling $\Phi$ to $G$,
a contradiction.
\end{proof}

A \emph{$k$-alternator} $( 3\leq k\leq\M )$ is a bipartite subgraph
$B(X,Y)$ of graph $G$ such that $d_B(x)=d_G(x)\leq k$ for each $x\in
X$ and $d_B(y)\geq d_G(y)+k-M$ for each $y\in Y$. This concept was
first introduced by Borodin, Kostochka and Woodall \cite{BKW} and
generalized by Wu and Wang \cite{WW}.

\begin{lem}[\cite{BKW}]\label{lem:3structure}
A bipartite graph $G$ is edge $f$-choosable where
$f(uv)=\max\{d(u),d(v)\}$ for any $uv\in E(G)$.
\end{lem}

\begin{lem}\label{lem:4structure}
There is no $k$-alternator $B(X,Y)$ in $G$ for any integer $k$ with
$3\leq k\leq\M$.
\end{lem}

\begin{proof}
Suppose that there exits a $k$-alternator $B(X,Y)$ in $G$.
Obviously, $X$ is an independent set of vertices in graph $G$ by
Lemma \ref{lem:2structure}. By the minimality of $G$, the subgraph
$G[V(G)\backslash X]$ has a (2,1)-total labelling $\Phi$ with color
interval $C$. Then for each $xy\in B(X,Y)$,
$|A_\Phi(xy)|\geq|C|-|F_\Phi(y)|-|F_\Phi(x)|\geq
|C|-\left(d_G(y)-d_B(y)+3\right)-0\geq
M+3-\left(M-d_B(y)+3\right)\geq d_B(y)$ and $|A_\Phi(xy)|\geq
|C|-\left(d_G(y)-d_B(y)+3\right)\geq M+3-(M+3-k)\geq k$ because
$B(X,Y)$ is a $k$-alternator. Therefore,
$|A(xy)|\geq\max\{d_B(y),d_B(x)\}$. By Lemma \ref{lem:3structure},
it follows that $E(B(X,Y))$ can be colored properly. Denote this new
partial (2,1)-total labelling by $\Phi'$. Then for each vertex $x\in
X$, $|A_{\Phi'}(x)|\geq|C|-|\cup_{z\in N(x)}\Phi'(z)|-|\cup_{e\ni
x}I_{\Phi'}(e)|\geq |C|-4d(x)\geq M+3-(M+2)\geq1$ because
$d_G(x)\leq k\leq\M$. Thus, we can extend the partial (2,1)-total
labelling  $\Phi$ to $G$, a contradiction.
\end{proof}

\begin{lem}\label{lem:5structure}
Let $X_k=\{x\in V(G)\ \big|\ d_G(x)\leq k\}$ and $Y_k=\cup_{x\in
X_k}N(x)$ for any integer $k$ with $3\leq k\leq\M$. If
$X_k\neq\emptyset$, then there exists a bipartite subgraph $M_k$ of
$G$ with partite sets $X_k$ and $Y_k$ such that $d_{M_k}(x)=1$ for
each $x\in X_k$ and $d_{M_k}(y)\leq k-1$ for each $y\in Y_k$.
\end{lem}

\begin{proof}
The proof is omitted here since it is similar with the proof of
Lemma 2.4 in Wu and Wang \cite{WW}.
\end{proof}

We call $y$ the $k$-master of $x$ if $xy\in M_k$ and $x\in X_k, y\in
Y_k$. By Lemma \ref{lem:2structure}, if $uv\in E(G)$ satisfies
$d(v)\leq\M$ and $d(u)=M-i$, then $d(v)\geq M+2-d(u)\geq i+2$.
Together with Lemma \ref{lem:5structure}, it follows that each
$(M-i)$-vertex can be a $j$-master of at most $j-1$ vertices, where
$2\leq i+2\leq j\leq\M$. Each $i$-vertex has a $j$-master where
$2\leq i\leq j\leq\M$.

\begin{lem}\label{lem:6structure}
$G$ has the following structural properties.
\begin{itemize}
  \vspace{-2mm}\item [(a)]\ A 4-vertex is adjacent to $8^+$-vertices;

  \vspace{-2mm}\item [(b)]There is no $[d(v_1), d(v_2), d(v_3)]$-face with $d(v_1)=5$,
  $\max\set{d(v_2),\, d(v_3)}\leq6$;

  \vspace{-2mm}\item [(c)]\ If $f=[d(v_1), d(v_2), d(v_3)]$ is a triangle face with $d(v_1)=5$,
  $d(v_2)=6$ and $d(v_3)=7$, then $v_1$ has no other 6-neighbors
  besides $v_2$.

  \vspace{-2mm}\item [(d)]\ If a vertex $v$ is adjacent to two vertices $v_1$,
  $v_2$ such that $2\leq d(v_1)=d(v_2)=M+2-d(v)\leq3$, then every
  face incident with $vv_1$ or $vv_2$ must be a $4^+$-face.

  \vspace{-2mm}\item [(e)]\ Each $\Delta$-vertex can be adjacent to at most one
  2-vertex.
\end{itemize}

\begin{figure}
\begin{center}
  \includegraphics[width=11cm,height=3.5cm]{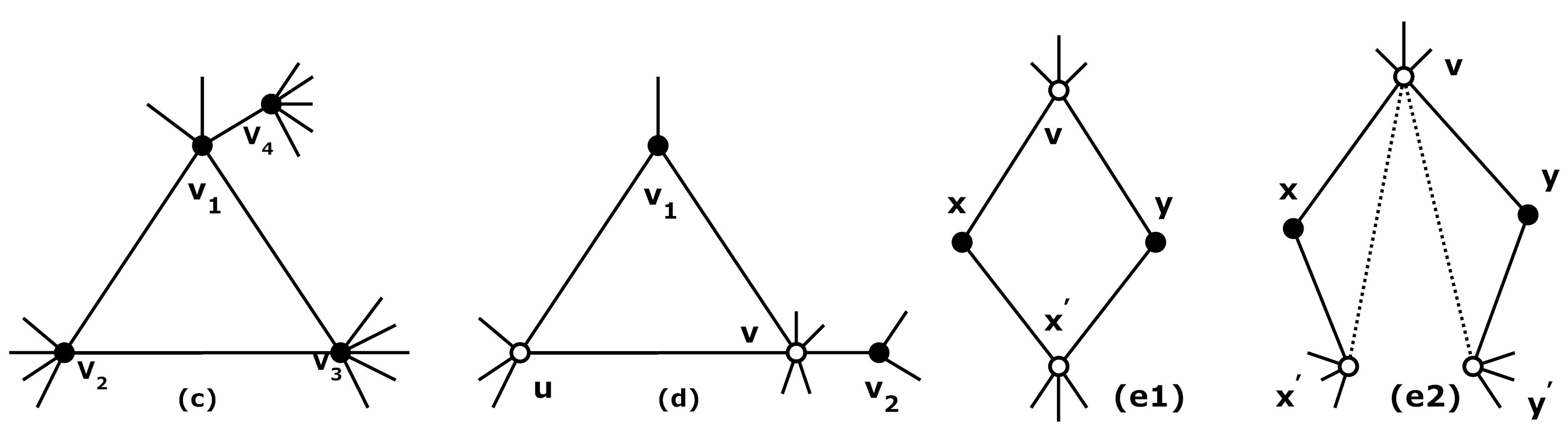}\\
  \caption{Reducible configurations of Lemma \ref{lem:6structure}.}\label{Fig:RC}
\end{center}
\end{figure}

\end{lem}

\begin{proof} {\large (\textbf{\textit{a}})}\, Otherwise, suppose that there is $uv\in E$ such that $d(u)=4$ and $d(v)\leq7$. By the minimality of $G$, $H=G-uv$ is (2,1)-total labelable with color interval $C$. Erase the color of vertex $u$, and denote this partial (2,1)-total labelling by $\Phi$. Then $|A_\Phi(u)|\geq |C|-|\F|-|\E|\geq M+3-4-3\times3\geq15-13\geq2$ and $|A_\Phi(uv)|\geq|C|-|E_\Phi(u)|-|F_\Phi(v)|\geq15-3-(6+3)\geq3$. Choose $\alpha\in A_\Phi(u)$ to color $u$. If $A_\Phi(uv)\neq\{\alpha-1, \alpha, \alpha+1\}$, then we can choose $\gamma\in A_\Phi(uv)\backslash\{\alpha-1, \alpha,\alpha+1\}$ to color edge $uv$. Otherwise, $A_\Phi(uv)=\{\alpha-1,\alpha, \alpha+1\}$. Then we choose $\beta\in A_\Phi(u)\backslash \{\alpha\}$ to color $u$. Since $A_\Phi(uv)\neq\{\beta-1, \beta,\beta+1\}$, we can choose $\gamma'\in A_\Phi(uv)\backslash\{\beta-1,\beta, \beta+1\}$ to color edge $uv$. Thus, we extend $\Phi$ from subgraph $H$ to $G$, a contradiction.

{\large (\textbf{\textit{b}})}\, By Lemma \ref{lem:1structure}, it
is enough to prove that there is no [5, 6, 6]-face. Otherwise, let
$d(v_1)=5$, $d(v_2)=d(v_3)=6$ and let $H=G-\{v_1v_2, v_1v_3\}$. Then
$H$ has a (2,1)-total labelling $\Phi$ with interval $C$.

{\bf Case 1.}\ $\Phi(v_1)\in F_\Phi(v_2)\cup F_\Phi(v_3)$. Without
loss of generality, suppose that $\Phi(v_1)\in F_\Phi(v_2)$, \ie
$|F_\Phi(v_1)\cup F_\Phi(v_2)|\leq|F_\Phi(v_1)|+|F_\Phi(v_2)|-1$.
Then $|A_\Phi(v_1v_2)|=M+3-|F_\Phi(v_1)\cup F_\Phi(v_2)|\geq
15-(3+3+5+3-1)\geq2$ and $|A_\Phi(v_1v_3)|=M+3-|F_\Phi(v_1)\cup
F_\Phi(v_3)|\geq 15-(3+3+5+3)\geq1$ which implies that we can extend
$\Phi$ to $G$, a contradiction.

{\bf Case 2.}\ $\Phi(v_1)\notin F_\Phi(v_2)\cup F_\Phi(v_3)$. That
is, $\Phi(v_1)\in A_\Phi(v_2v_3)$. Recolor $v_2v_3$ with color
$\Phi(v_1)$ and denote this new partial (2,1)-total labelling by
$\Phi'$. Then $|F_{\Phi'}(v_1)\cup
F_{\Phi'}(v_2)|\leq|F_{\Phi'}(v_1)|+|F_{\Phi'}(v_2)|-1$. Analogous
to Case 1, we can extend $\Phi'$ to $G$, a contradiction.

{\large (\textbf{\textit{c}})}\, Suppose on the contrary that $G$
contains such a configuration (see Fig. \ref{Fig:RC} (c)). By the
minimality of $G$ $H=G-\{v_1v_2, v_1v_3\}$ has a (2,1)-total
labelling $\Phi$ with color interval $C$.

{\bf Claim 1.}\ $\Phi(v_1)\neq\Phi(v_2v_3)$. Otherwise,
$|F_\Phi(v_1)\cap F_\Phi(v_2)|\geq1$. Then
$|A_\Phi(v_1v_2)|=M+3-|F_\Phi(v_1)\cup F_\Phi(v_2)|\geq
15-(3+3+5+3-1)\geq2$ and $|A_\Phi(v_1v_3)|=M+3-|F_\Phi(v_1)\cup
F_\Phi(v_3)|\geq 15-(3+3+6+3-1)\geq1$ which implies that we can
extend the partial (2,1)-total labelling $\Phi$ to $G$, a
contradiction.

{\bf Claim 2.}\ $E_\Phi(v_1)\subseteq F_\Phi(v_2)\cup F_\Phi(v_3)$.
Otherwise, we can choose a color $\alpha\in E_\Phi(v_1)\backslash
\lk F_\Phi(v_2)\cup F_\Phi(v_3)\rk\neq \emptyset$ to recolor edge
$v_2v_3$. Denote this new coloring of $H$ by $\Phi'$. Then
$\alpha\in F_{\Phi'}(v_1)\cap F_{\Phi'}(v_2)\cap F_{\Phi'}(v_3)$.
Therefore, $|A_{\Phi'}(v_1v_2)|\geq2$ and $|A_{\Phi'}(v_1v_3)|\geq1$
which implies that we can extend $\Phi'$ to $G$, a contradiction.

{\bf Claim 3.}\ $E_\Phi(v_1)\subseteq F_\Phi(v_2)$. If not, we have
$E_\Phi(v_1)\cap F_\Phi(v_3)\neq\emptyset$ by Claim 2. Assume that
$E_\Phi(v_1)\subseteq F_\Phi(v_3)$. Then $|F_\Phi(v_1)\cap
F_\Phi(v_3)|\geq3$, which implies that
$|A_\Phi(v_1v_3)|=M+3-|F_\Phi(v_1)\cup F_\Phi(v_3)|\geq
15-(3+3+6+3-3)\geq3$ and $|A_\Phi(v_1v_2)|=M+3-|F_\Phi(v_1)\cup
F_\Phi(v_2)|\geq 15-(3+3+5+3)\geq1$. Therefore, we can extend $\Phi$
from subgraph $H$ to $G$, a contradiction.

By Claim 2 and Claim 3, we have $|F_\Phi(v_1)\cap
F_\Phi(v_2)|\geq|E_\Phi(v_1)\cap F_\Phi(v_2)|=3$ and
$E_\Phi(v_1)\cap F_\Phi(v_3)=\emptyset$. Since $\Phi(v_1v_4)\in
E_\Phi(v_1)$, we have $\Phi(v_1v_4)\notin F_\Phi(v_3)$. For edge
$v_1v_4$, $|A_\Phi(v_1v_4)|=M+3-|F_\Phi(v_1)\cup F_\Phi(v_4)|\geq
15-(3+3+5+3)\geq1$. Therefore, we choose $\alpha\in A_\Phi(v_1v_4)$
to recolor $v_1v_4$ and denote this new partial (2,1)-total
labelling by $\Phi'$. Obviously,
$F_{\Phi'}(v_1)=F_\Phi(v_1)\cup\{\alpha\}\backslash\{\Phi(v_1v_4)\}$,\
$F_{\Phi'}(v_2)=F_\Phi(v_2)$ and $F_{\Phi'}(v_3)=F_\Phi(v_3)$. Thus,
$\Phi(v_1v_4)\notin F_{\Phi'}(v_1)\cup F_{\Phi'}(v_3)$ which implies
that we can color $v_1v_3$ with $\Phi(v_1v_4)$. For edge $v_1v_2$,
we have $|A_{\Phi'}(v_1v_2)|=M+3-|F_{\Phi'}(v_1)\cup
F_{\Phi'}(v_2)|\geq 15-(3+3+5+3-2)\geq3$ because
$|F_{\Phi'}(v_1)\cap F_{\Phi'}(v_2)|\geq2$. Therefore, we choose
$\Phi(v_1v_4)$ and $\beta\in A_{\Phi'}(v_1v_2)\backslash
\{\Phi(v_1v_4)\}$ to color $v_1v_3$ and $v_1v_2$, respectively. Then
we obtain a (2,1)-total labelling of $G$, a contradiction.

{\large (\textbf{\textit{d}})}\, Assume that there is a triangle
face $f=uvv_1$ such that $vv_2\in E(G)$ and $2\leq
d(v_1)=d(v_2)=M+2-d(v)\leq3$ (see Fig. \ref{Fig:RC} (d)). By the
minimality of $G$, $H=G-\{vv_1, vv_2\}$ is (2,1)-total labelable
with color interval $C$. Erase the colors of $v_1$ and $v_2$, and
denote this partial (2,1)-total labelling by $\Phi$. Then
$|A_\Phi(vv_1)|\geq M+3-|E_\Phi(v_1)|-|F_\Phi(v)|\geq
M+3-(d(v_1)-1)-(d(v)-2+3)\geq M+3-(d(v_1)+d(v))\geq1$. Similarly,
$|A_\Phi(vv_2)|\geq1$.

{\bf Case 1.}\ $\max\set{|A_\Phi(vv_1)|,\, |A_\Phi(vv_2)|}\geq2 $ or
$A_\Phi(vv_1)\neq A_\Phi(vv_2)$. Then we can label $vv_1$ and $vv_2$
properly by choosing colors from $A_\Phi(vv_1)$ and $A_\Phi(vv_2)$,
respectively.

{\bf Case 2.}\ Otherwise, suppose that
$A_\Phi(vv_1)=A_\Phi(vv_2)=\set \alpha$. Then $\lk E_\Phi(v_1)\cup
E_\Phi(v_2)\rk\cap F_\Phi(v)=\emptyset$ and
$E_\Phi(v_1)=E_\Phi(v_2)$. Therefore we can exchange the colors of
$uv_1$ and $uv$. Denote this new partial (2,1)-total labelling by
$\Phi'$. It is not difficult to see that
$A_{\Phi'}(vv_1)=A_{\Phi}(vv_1)=\set \alpha$ and
$|A_{\Phi'}(vv_2)|\geq M+3-(|E_{\Phi'}(v_2)|+|F_{\Phi'}(v)|-1)\geq
M+3-(d(v_2)+d(v)-1)\geq2$, which is similar to Case 1.

Now consider the 2-vertices $v_1$ and $v_2$. $|A_{\Phi}(v_1)|\geq
M+3-4d(v_1)\geq M-9\geq3$ and $|A_{\Phi}(v_2)|\geq M+3-4d(v_2)\geq3$
which implies that we can obtain a (2,1)-total labelling of $G$, a
contradiction.

{\large (\textbf{\textit{e}})}\, Suppose that $v$ is a
$\Delta$-vertex adjacent to two 2-vertices $x$ and $y$. Let $x'$
(resp. $y'$) be the neighbor of $x$ (resp. $y$) different from $v$.

{\bf Case 1.}\ $x'= y'$, \ie $vxx'y$ forms a 4-cycle (see Fig.
\ref{Fig:RC} (e1)). By the minimality of $G$, $H=G-\{x, y\}$ has a
(2,1)-total labelling $\Phi$ with color interval $C$. Then
$|A_\Phi(vx)|\geq M+3-|F_\Phi(v)|\geq M+3-(\Delta-2+3)\geq2$.
Similarly, $|A_\Phi(vy)|\geq2$, $|A_\Phi(x'x)|\geq2$,
$|A_\Phi(x'y)|\geq2$. Since $\chi'_l(C_{4})=2$, we can choose colors
to label all the edges of 4-cycle $vxx'y$ properly. Denote this new
partial (2,1)-total labelling by $\Phi'$. Now consider the
2-vertices $x$ and $y$. Since $|A_{\Phi'}(x)|\geq M+3-4d(x)\geq
M-5\geq7$ and $|A_{\Phi'}(y)|\geq M+3-4d(y)\geq M-5\geq7$, we can
extend $\Phi'$ from $H$ to $G$, a contradiction.

{\bf Case 2.}\ $\set{vx', vy'}\cap E(G)\neq\emptyset$. This case is
impossible by Lemma \ref{lem:6structure} $(d)$.

{\bf Case 3.}\ $\set{vx', vy'}\cap E(G)=\emptyset$ (see Fig.
\ref{Fig:RC} (e2)). By the minimality of $G$,
$H=G-\set{x,y}\cup\set{vx', vy'}$ has a (2,1)-total labelling
$\Phi$, which implies that $\Phi(vx')\notin F_\Phi(x')\cup
F_\Phi(v)$ and $\Phi(vy')\notin F_\Phi(y')\cup F_\Phi(v)$. Color
$xx', vy$ with $\Phi(vx')$ and color $yy', vx$ with $\Phi(vy')$.
Then we obtain a partial (2,1)-total labelling $\Phi'$ of $G$. Since
$|A_{\Phi'}(x)|\geq M+3-4d(x)\geq M-5\geq7$ and $|A_{\Phi'}(y)|\geq
M+3-4d(y)\geq M-5\geq7$, we can choose colors to label $x$ and $y$
properly. Thus we extend $\Phi'$ to graph $G$, a contradiction.
\end{proof}

In the next section, we call a $[5, 6, 7]$-face be a \emph{special
3-face} and the other 3-face be a \emph{normal 3-face}. Lemma
\ref{lem:6structure} implies that each 5-vertex has at most two
special 3-faces incident with it.

\section{Discharging Part}
\begin{proof} [\noindent\textit{\textbf{Proof of Theorem
\ref{Thm:new2TL}.}}]

Let $G$ be a minimal counterexample in terms of $|V|+|E|$ with
$M\geq 12$. By the Lemmas of Section 2, $G$ has structural
properties in the following.
\begin{itemize}

\vspace{-3mm}\item[$(C1)$] $G$ is connected;

\vspace{-3mm}\item[$(C2)$] For each $uv\in E$, $d(u)+d(v)\geq M-1$;

\vspace{-3mm}\item[$(C3)$] If $uv\in E$ and
$\min\{d(u),d(v)\}\leq\M$, then $d(u)+d(v)\geq M+2$.

\vspace{-3mm}\item[$(C4)$] Each $i$-vertex (if exists) has one
$j$-master, where $2\leq i\leq j\leq3$;

\vspace{-3mm}\item[$(C5)$] Each $(M-i)$-vertex (if exists) can be a
$j$-master of at most $j-1$ vertices, where $2\leq i+2\leq j\leq3$.

\vspace{-3mm}\item[$(C6)$] $G$ satisfies $(a)$--$(e)$ of  Lemma
\ref{lem:6structure}.

\end{itemize}

We define the initial charge function $w(x):=d(x)-4$ for all element
$x\in V\cup F$. By Euler's formula $|V|-|E|+|F|=2$, we have
$\sum\limits_{x\in V\cup F}w(x)=\sum\limits_{v\in
V}(d(v)-4)+\sum\limits_{f\in F}(d(f)-4)=-8<0$. The discharging rules
are defined as follows.

\textit{$(R1)$}\  Each 2-vertex receives charge $\DF{1}{2}$ from
each of its incident $\Delta$-vertex and receives charge 1 from its
3-master.

\textit{$(R2)$}\  Each 3-vertex receives charge 1 from its
$3$-master.

\textit{$(R3)$}\  Each 5-vertex transfer charge $\frac{1}{4}$ to
each of its incident special 3-face and transfer $\DF{1}{6}$ to each
of its incident normal 3-face.

\textit{$(R4)$}\  Each $k$-vertex with $6\leq k\leq7$ transfer
charge $\DF{k-4}{k}$ to each 3-face that incident with it.

\textit{$(R5)$}\  Each $8^+$-vertex transfer charge $\DF{1}{2}$ to
each 3-face that incident with it.

Let $v$ be a $k$-vertex of $G$. If $k=2$, then $w'(v)=
w(v)+1+\DF{1}{2}\times2=-2+1+1=0$ by $(R1)$ and $(C3)$; If $k=3$,
then $w'(v)= w(v)+1=0$ since it receives 1 from its 3-master by
$(R2)$ and $(C4)$; If $k=4$, then $w'(v)= w(v)=0$ since we never
change the charge by our rules; If $k=5$, then $w'(v)\geq
w(v)-\frac{1}{4}\times2-\frac{1}{6}\times3=0$ by $(R3)$ and Lemma
\ref{lem:6structure} (c); If $6\leq k\leq7$, then $w'(v)\geq
w(v)-k\DF{k-4}{k}=0$ by $(R4)$; If $8\leq k\leq M-2$, then
$w'(v)\geq w(v)-k\DF{1}{2}\geq0$ by $(R5)$ and $(C3)$;

By Lemma \ref{lem:2structure}, it is not difficult to prove that
$\delta(G)\geq2$ when $\Delta=M$ and $\delta(G)\geq3$ otherwise. If
$M\geq\Delta+2$, then $M-2\geq\Delta$. Thus, $w(v)\geq0$ for all
$v\in V(G)$. Otherwise, $\Delta\leq M\leq \Delta+1$. Consider the
$k$-vertex $v$ with $M-1\leq d(v)=k\leq\Delta$.

{\bf Case 1.}\ $M=\Delta+1$. Then $\delta\geq3$ and $k=\Delta=M-1$.
Lemma \ref{lem:6structure} $(d)$ implies that $(M-1)$-vertex is
incident with at most $M-4$ triangle faces if it has at least two
3-neighbors. Thus, together with rules $(R2)$ and $(R5)$, we have
$w'(v)\geq \min\set
{w(v)-\DF{1}{2}\Delta-1,\,w(v)-\DF{1}{2}(\Delta-3)-2}=\DF{M-1}{2}-5\geq
\DF{1}{2}$.

{\bf Case 2.}\ $M=\Delta$. Then $\delta\geq2$. For $k=\Delta-1=M-1$,
see Case 1. For $k=\Delta=M$, $w'(v)\geq
\min\set{w(v)-\DF{1}{2}\Delta-1-\DF{1}{2},\,
w(v)-\DF{1}{2}(\Delta-3)-2-\DF{1}{2}}=\DF{M-11}{2}>0$ by Lemma
\ref{lem:6structure} $(d)$, $(e)$ and rules $(R1)$, $(R2)$, $(R5)$.

Let $f$ be a $k$-face of $G$. If $k\geq4$, then $w'(f)=w(f)\geq0$
since we never change the charge of them by our rules; If $k=3$,
assume that $f=[d(v_1),d(v_2),d(v_3)]$ with $d(v_1)\leq d(v_2)\leq
d(v_3)$. It is easy to see $w(f)=-1$. Consider the subcases as
follows.

$(1)$\ Suppose $d(v_1)\leq3$. Then $\min\set{d(v_2),d(v_3)}\geq
M+2-d(v_1)\geq M-1\geq11$ by $(C3)$. Thus,
$w'(f)=w(f)+\DF{1}{2}\times2=0$ by $(R5)$;

$(2)$\ Suppose $d(v_1)=4$. Then $d(v_3)\geq d(v_2)\geq8$ by Lemma
\ref{lem:6structure} (a). Therefore, $w'(f)=w(f)+\DF{1}{2}\times2=0$
by $(R5)$;

$(3)$\ Suppose $d(v_1)=5$. Then $d(v_2)=6, d(v_3)\geq7$ or
$d(v_3)\geq d(v_2)\geq7$ by Lemma \ref{lem:6structure} (b). If $f$
is a special 3-face, then $w'(f)\geq
w(f)+\DF{1}{4}+\DF{1}{3}+\DF{3}{7}=\DF{1}{84}>0$ by $(R2)$ and
$(R3)$. If $f$ is a normal 3-face, then $d(v_2)=6, d(v_3)\geq8$ or
$d(v_3)\geq d(v_2)\geq7$. Therefore, $w'(f)\geq
w(f)+\DF{1}{6}+\min\{\DF{1}{3}+\DF{1}{2}, \DF{3}{7}\times2\}\geq0$
by $(R3)-(R5)$.

$(4)$\ Suppose $d(v_1)=m\geq6$. Then $d(v_3)\geq d(v_2)\geq6$.
Therefore, $w'(f)\geq w(f)+3\times\min\{\DF{m-4}{m},\DF{1}{2}\}=0$
by $(R4)$ and $(R5)$.

Thus, we have $\sum\limits_{x\in V\cup F}w(x)=\sum\limits_{x\in
V\cup F}w'(x)>0$ since $w(v)>0$ if $d(v)=\Delta$ and this
contradiction completes our proof.
\end{proof}

Actually, we show that $\sum\limits_{x\in V\cup
F}w(x)=\sum\limits_{x\in V\cup F}w'(x)>0$ in our proof of Theorem
\ref{Thm:new2TL}. It implies the following corollary immediately.

\begin{cor}
Let $G$ be a graph embedded in a surface of nonnegative Euler
characteristic with maximum degree $\Delta\geq12$. Then
$\lambda_2^T(G)\leq\Delta+2$.
\end{cor}




\end{document}